\documentclass[11pt]{amsart}
\usepackage{amsmath,amssymb,amsthm}
\usepackage{amsfonts}
\usepackage[mathscr]{eucal}
\usepackage{amsmath,amssymb,amsthm}
\usepackage{amsfonts}
\usepackage{color}
\usepackage[mathscr]{eucal}
\newtheorem{theorem}{Theorem}[section]
\newtheorem{lemma}[theorem]{Lemma}

\newtheorem{corollary}[theorem]{Corollary}

\theoremstyle{definition}

\theoremstyle{remark}

\numberwithin{equation}{section}

\newcommand{\rr}{{\mathbb R}}

\def\R{{\mathbb R}}
\def\N{{\mathbb N}}

\def\l{{\langle}}
\def\r{\rangle}
\def\t{\tau}
\def\k{\kappa}
\def\a{\alpha}

\def\eps{\varepsilon}

\def\Re {{\rm Re}\,}

\def\E{{\mathbb E}}
\def\P{{\mathbb P}}

\def\k{{\bf k}}
\def\n{{\bf n}}
\def\m{{\bf m}}
\def\0{{\bf 0}}
\def\s{{\bf s}}
\def\t{{\bf t}}

\begin{document}

\title{\bf Strong laws of large numbers for arrays of random variables
and stable random fields}

\author{Erkan Nane}
\address{Erkan Nane, Department of Mathematics and Statistics,
Auburn University, Auburn, AL 36849}
\email{nane@auburn.edu}
\urladdr{http://www.auburn.edu/$\sim$ezn0001}

\author{Yimin Xiao}
\address{Yimin Xiao, Department Statistics and Probability,
Michigan State University, East Lansing, MI 48824}
\email{xiao@stt.msu.edu}
\urladdr{http://www.stt.msu.edu/$\sim$xiaoyimi}
\thanks{Research of Y. Xiao was partially supported by
grants from the National Science Foundation.}

\author{Aklilu Zeleke}
\address{Aklilu Zeleke, Department of Statistics and Probability,
Michigan State University, East Lansing, MI 48824}
\email{zeleke@stt.msu.edu}

\begin{abstract}
Strong laws of large numbers are established for random fields
with weak or strong dependence. These limit theorems are
applicable to random fields with heavy-tailed distributions
including fractional stable random fields.

The conditions for SLLN are described in terms of the $p$-th
moments of the partial sums of the random fields, which are
convenient to verify. The main technical tool in this paper
is a maximal inequality for the moments of partial sums of
random fields that extends the technique of Levental, Chobanyan
and Salehi \cite{chobanyan-l-s} for a sequence of random variables
indexed by a one-parameter.

\end{abstract}


\maketitle


\section{Introduction}

Many authors have studied strong laws of large numbers (SLLN) for arrays
of random variables or, more generally, random fields with certain dependence
structures. For example, Klesov \cite{K82,Klesov03} proved a strong law of
large numbers for orthogonal random fields and related asymptotic properties.
M\'oricz \cite{Moricz-78,Moricz-79, Moricz-80,Moricz-87, Moricz-89, Moricz-91}
established SLLN for quasi-orthogonal or quasi-stationary random fields.
M\'oricz, Stadtm\"uller and Thalmaier \cite{MST08} proved SLLN for blockwise
${\mathcal M}$-dependent random fields under moment conditions. Thanh \cite{thanh}
provided a necessary and sufficient condition for a general $d$-dimensional
arrays of random variables to satisfy a strong law of large numbers, which
can be applied to blockwise independent random fields.

For random fields with more information on their dependence structures such
as linear random fields, more results on their limiting behaviors have been known.
For example, Marinucci and Poghosyan \cite{MarinucciPoghosyan}, and Paulauskas
\cite{Paul} studied the asymptotics for linear random fields,
including law of large numbers, central limit theorems and invariance principles,
by applying the Beveridge-Nelson decomposition method. By applying ergodic
theory, Banys et al. \cite{BDP} studied strong law of large numbers of linear
random fields generated by ergodic or mixing random variables. Recently, Sang
and Xiao \cite{SangXiao} proved exact moderate and large deviation results for
linear random  fields with independent innovations and applied them to establish
laws of the iterated logarithm for linear random fields.

This paper is mainly motivated by our interest in asymptotic properties
of fractional random fields with long-range dependence and/or heavy-tailed
distributions. An important class of such random fields is formed by fractional
stable fields (cf. e.g., Samorodnitsky and Taqqu \cite{ST94}, Cohen and Istas
\cite{CI13}, Pipiras and Taqqu \cite{PT17}, Ayache, Roueff and Xiao \cite{ARX08},
Xiao \cite{Xiao08}).

We start with some notation and definitions. Let $\N_0=\N\cup \{0\}$ be the set
of non-negative integers. For any $\n \in \N_0^d$, we write it as $\n= (n_1, \ldots, n_d)$
or $\n = \langle n_i\rangle$. The $\ell^2$-norm and $\ell^\infty$-norm of $\n$ are
defined by
$$
\|\n\| = \big(n_1^2 + \cdots+ n_d^2\big)^{1/2}, \qquad
|\n| = \max\{n_1,n_2,\cdots, n_d\},$$
respectively.  For any constant $a > 0$, we denote
$a^{\n} = (a^{n_1}, \ldots, a^{n_d})$.

There is a natural partial order in $\N_0^d$: $\m \le \n$ if and only if
$m_i \le n_i$ for all $i = 1, \ldots, d$. For any $\m \le \n$, we denote $[\m, \n]
= \{ \k \in \N_0^d: \m \le \k \le \n\}$, which is called an interval
(or a rectangle) in $\N_0^d$.

Let $\{\xi(\n), \n \in \N_0^d\}$ be a real-valued random field indexed
by the lattice $\N_0^d$ and let $\{\Gamma_\n \}$ be a sequence of increasing
subsets of $ \N_0^d$. Denote the partial sum of $\{\xi(\n), \n \in \N_0^d\}$ over
$\Gamma_\n$ by
$$
S(\Gamma_\n) = \sum_{\k\in \Gamma_\n}\xi_{\k}.
$$
Let $\varphi: \N^d \to \R_+$ be a function such that $\varphi(\n) \to \infty$
as $|\n| \to \infty.$ 
We say that a random field $\{\xi(\n), \n \in \N_0^d\}$ satisfies the
strong law of large numbers with respect to $\{\Gamma_\n \}$ and $\varphi$
if
\begin{equation}\label{Def:SLLN}
\lim_{|\n| \to \infty} \frac{S(\Gamma_\n)}
{\varphi(\n)}  = 0\qquad \hbox{ a.s.}
\end{equation}
In the random field setting, the index sets $\Gamma_\n$ can have various 
configurations, ranging from spherical to rectangular, and to non-standard 
shapes, which arise naturally in many applied areas such as spatial statistics.
When $\{\Gamma_\n \}$ is clear from the context, we will simply refer to
(\ref{Def:SLLN}) as a $\varphi$-SLLN.

In this paper, we will consider spherical and rectangular sets  $\{\Gamma_\n \}$, 
see below. M\'oricz \cite{Moricz-78} considered SLLN over more types of increasing 
domains in $\N_0^d$. 

Two types of functions $\varphi$ are of particular interest in this paper:
\begin{itemize}
\item[(i)]\ $\varphi(\n) = f(\|\n\|)$, where $f:\R_+\to \R_+$ is a
non-decreasing function such that $f(x) \to \infty$ as $x \to \infty$. Such
functions $\varphi$ are useful in establishing SLLN for (approximately) isotropic
random fields, including  a class of stable random fields
with stationary increments in \cite{Xiao08}.

\item[(ii)]\ $\varphi(\n)=\varphi_{1}(n_1) \cdots \varphi_d(n_d)$,
where $\varphi_{1}, \ldots, \varphi_{d}$ are non-decreasing functions on $\R_+$
such that $\varphi_i(x) \uparrow \infty$ as $x \to \infty$. Functions of
this form arise naturally in studying SLLN for anisotropic random
fields with multiplicative kernels. Typical examples are linear or harmonizable
fractional stable sheets (\cite{ARX08,Xiao08}).
\end{itemize}



Throughout this paper, $p>0$ is a constant. We first consider SLLN for partial sums
over spherical domains $\Gamma_\n = Q_{|\n|}$,  where for any $r\geq 1$,
$$
Q_r=\{ \k=(k_1,\cdots, k_d)\in \N_0^d:\ \|\k\| \leq r\}.
$$
If we replace the $\ell^2$-norm by the  $\ell^\infty$-norm $|\cdot|$, then $Q_r$ is the cube $[0, r]^d$.

Let $Q_0=\emptyset$. We denote
\begin{equation}\label{Eq:Sum_s}
S(m;n)=\sum_{\k \in Q_{n+m}\setminus Q_{m}}\xi_{\k}.
\end{equation}

The following SLLN holds under both $\ell^2$ and $\ell^\infty$-norms. It can be applied to isotropic random fields.

\begin{theorem}\label{main-thm1}
Let $f:\R_+\to \R_+$ be a non-decreasing function such that $f(x) \to \infty$ as $x \to \infty$.
Assume that there are
constants $C_{2}\geq C_{1}>1$ such that  $C_{1}\le f (2 { x})/f ({ x}) \le C_{2}$ for all $ x \in \R_+$.
Let $a \geq 2$ be an integer that satisfies  $C_{1}^{\lfloor \log_2 a\rfloor}
>\max\{a, a 2^{p-1}\}$.
If
\begin{equation}\label{Eq:MomCon2}
\sum_{n \in \N_0}
\sup_{m \in \N_0}\E\bigg(\frac{|S(m;a^n)|^p}
{f(a^{n})^p}\bigg)<\infty,
\end{equation}
then
\begin{equation}\label{limit-1-d-slln}
\lim_{n\to\infty }\frac{S(0;n)}{f(n)}=0,\ \ \mathrm{a.s}.
\end{equation}
\end{theorem}


The second theorem proves a SLLN for partial sums over rectangular domains $\Gamma_\n
= [{\bf 1},\, \n]$, that can be applied to anisotropic random fields.

For any $\m, \n \in \N_0^d$, denote by
\begin{equation}\label{Eq:Sum_rec}
S(\m;\, \n)=\displaystyle \sum_{k_1=m_1+1}^{m_1+n_1}
\cdots \sum_{k_d=m_d+1}^{m_d+n_d}
\xi(k_1, \cdots, k_d),
\end{equation}
the partial sum of the random variables $\xi(\k)$ over the interval
$(\m, \m+\n]$.

\begin{theorem}\label{main-thm2}
Suppose that, for $s=1,\cdots, d$,
$\varphi_s:\rr_+\to \rr_+$ is an increasing function that satisfy
$C_{3}<\varphi_s(2x)/\varphi_s(x)<C_{4}$ for all $x\in \R_+$
for some constants $C_{4}\geq C_{3}>1$.  Assume  $a\geq 2$
is an integer that satisfies   $C_{3}^{\lfloor \log_2 a\rfloor}
>\max\{a, a2^{p-1}\}$. If
\begin{equation}\label{Eq:MomCon1}
\sum_{\n \in \N_0^d}
\sup_{\m \in \N_0^d}\E\bigg(\frac{|S(\m; a^{n_1},
 \cdots, a^{n_d})|^p}{\varphi_1(a^{n_1})^p\cdots \varphi_d(a^{n_d})^p}\bigg)<\infty,
\end{equation}
then $\{\xi(\n), \n \in \N_0^d\}$ satisfies  the SLLN (\ref{Def:SLLN}) with $\Gamma_{\n} = [{\bf 1},\, \n]$ and
$\varphi(\n) = \varphi_1(n_1)\cdots \varphi_d(n_d)$.
\end{theorem}

The following direct consequence of Theorem \ref{main-thm2} is convenient to use.
\begin{corollary} \label{Co:2}
Let  the functions $\varphi_s$ ($s=1,\cdots, d$)  and the constant $a \ge 2$ be as in Theorem
\ref{main-thm2}. If there is  a  function $g$ such that
$$
\E \big(|S(\m;\, \n)|^p\big) \leq g(\n), \qquad \hbox{ for all }\ \m, \n \in \N_0^d
$$
and
$$
\sum_{\n \in \N_0^d} \frac{g(a^{n_1},\, \cdots, a^{n_d})}{\varphi_{1}(a^{n_1})^p
\cdots \varphi_{d}(a^{n_d})^p}<\infty,
$$
then  $\{\xi(\n), \, \n \in \N_0^d\}$ satisfies the SLLN as in Theorem
\ref{main-thm2}.
\end{corollary}

The method for proving Theorem \ref{main-thm2} is different from that
based on the Rademacher-Menshov-type maximal moment inequalities
(e.g., \cite{K82,Klesov03,Moricz-78,
Moricz-79, Moricz-80,Moricz-87, Moricz-89, Moricz-91,MST08,thanh}).
We rely more on the approach of Chobanyan et al.
\cite{ chobanyan-l-s,LevHC} and Nane et al. \cite{nane-xiao-zeleke}. In
particular, we  extend their maximal moment
inequalities  for sequences of random variables
to the case of random fields. We should mention that it is more
convenient to verify condition (\ref{Eq:MomCon1}) when the random field
$\{\xi(\n), \n \in \N_0^d\}$ has certain kind of stationarity (see Section
\ref{Sec:Appl}).


The rest of this paper is organized as follows. In Section \ref{sec:main1} and \ref{Sec:Main2},
we prove Theorems \ref{main-thm1} and \ref{main-thm2}, respectively.
Some applications of these theorems are given in
Section \ref{Sec:Appl}, where various random fields including fractional stable random fields,
orthogonal,  and quasi-stationary random fields are considered.

\section{\large Spherical sums: back to the one-dimensional case}
\label{sec:main1}

In this section we prove Theorem \ref{main-thm1}.
\begin{proof}[Proof of Theorem \ref{main-thm1}]
Recall $S(m;n)$ in (\ref{Eq:Sum_s}), and let
$ M(m;n)= \max_{k \le n} |S(m; k)|.
$
We divide the proof into two steps.

\noindent {\bf Step 1:} Our first task is to establish  a recursion for $ M(m;n)$.
We first consider the case $p>1$.
For any $k\in \N_0, n\in\N_0$, notice that
\begin{equation*}\label{Eq:Mkp}
M(k; a^{n+1})\leq \max\Big\{M(k; (a-1)a^{n}),\, |S(k;(a-1)a^n)|+M(k+(a-1)a^n; a^{n})\Big\}
\end{equation*}
 and in general for $l=2,\cdots, a$,
 \begin{equation*}\label{Eq:Mlp2}
M(k; la^{n})\leq \max\Big\{M(k; (l-1)a^{n}), \,|S(k;(l-1)a^n)|+M(k+(l-1)a^n; a^{n})\Big\}.
\end{equation*}
By using the elementary inequality $|x+y|^p\leq 2^{p-1}(|x|^p+|y|^p)$  we get
\begin{equation}
\begin{split}
&M^p(k; la^{n})\\
&\leq \max\Big\{M^p(k; (l-1)a^{n}), \, 2^{p-1}(|S(k;(l-1)a^n)|^p
+M^p(k+(l-1)a^n; a^{n}))\Big\}  \\
&\leq  (2^{p-1}-1)|S(k;(l-1)a^n)|^p+M^p(k; (l-1)a^{n})+2^{p-1}M^p(k+(l-1)a^n; a^{n}).
\label{max-1}
\end{split}
\end{equation}
Eq. \eqref{max-1} can be written as
\begin{equation} \label{Mlp-recursion}
\begin{split}
&M^p(k; la^{n})-|S(k;la^{n})|^p \\
&\leq   M^p(k; (l-1)a^{n})-|S(k,(l-1)a^n)|^p \\
&\quad +2^{p-1}\Big(M^p(k+(l-1)a^n; a^{n})-|S(k+(l-1)a^n;a^n)|^p\Big) \\
&\quad  -|S(k;la^{n})|^p + 2^{p-1}|S(k;(l-1)a^n)|^p+2^{p-1}|S(k+(l-1)a^n;a^n)|^p.
 \end{split}
 \end{equation}
For $l=1,\cdots,a$, define
\begin{equation}\label{Eq:Fn}
F_l(n)=\sup_{k\in \N_0}\E\bigg(\frac{M^p(k; la^{n})-|S(k;la^n)|^p}{f(la^{n})^p}\bigg) \ \
\hbox{ and }\  \ F (n+1) := F_a(n).
\end{equation}
Here 
\begin{equation}\label{Eq:Fn}
F(0)=\sup_{k\in \N_0}\E\bigg(\frac{M^p(k; 1)-|S(k;1)|^p}{f(1)^p}\bigg)
\end{equation}
By  the definition of $M(k; 1)$, $M(k; 1)=|S(k;1)|$ in  both $\ell^2$-norm and   $\ell^\infty$-norm. Hence $F(0)=0$.

Dividing both sides of \eqref{Mlp-recursion} by $f(la^{n})^p$, taking expectations, and
then the supremum over all the $k$'s, we get
\begin{eqnarray}\label{Eq:2.6}
F_l(n)\leq \frac{f((l-1)a^n)^p}{f(la^n)^p}F_{l-1}(n)+2^{p-1}\frac{f(a^n)^p}{f(la^n)^p}F_1(n)+G_{l}(n),
\end{eqnarray}
where  $G_{l}(n)$ is
\[
\begin{split}
G_l(n)&=\sup_{k\in \N_0}\Bigg\{\frac{2^{p-1}f(a^{n})^p}{f(la^{n})^p}
\E\bigg(\bigg|\frac{S(k+(l-1)a^n,a^n)}{f(a^{n})}\bigg|^p\bigg)\\
& \qquad \qquad \quad +
\frac{2^{p-1}f((l-1)a^{n})^p}{f(la^{n})^p}\E\bigg(\bigg|\frac{S(k,(l-1)a^n)}
{f((l-1)a^{n})}\bigg|^p\bigg)
- \E\bigg(\bigg|\frac{S(k;la^{n})}{f(la^{n})}\bigg|^p\bigg)\Bigg\}.
\end{split}
\]

The terms in $G_l(n)$ contribute nicely in the recursion for $F(n)$ because of the following observation.
By the assumption, we have
$$
\frac{f(a^{n})^p}{f(a^{n+1})^p}(1+(a-1)2^{p-1})\leq\frac{(1+(a-1)2^{p-1})}
{C_1^{p\lfloor \log_2a\rfloor}} := c<1.
$$
Let $D_{a,p}=2^{p-1}\big[(a-1)^{p-1}+(a-2)^{p-1}+\cdots+1+(a-1)\big]$. By
iterating (\ref{Eq:2.6}) as in the proof of Theorem 2.1 in
\cite{nane-xiao-zeleke} we obtain the following recursion:
\begin{equation}\label{the-recursion}
F_a(n)=F(n+1)\leq c F(n)+D_{a,p}\,\sup_{k\in \N_0}
\E\Big|\frac{S(k;a^n)}{f(a^n)}\Big|^p.
\end{equation}

\noindent {\bf Step 2:} We use the recursion obtained in Step 1 to finish the proof.
By summing up \eqref{the-recursion} from $n=0$ to $\infty$ and
using (\ref{Eq:MomCon2}) we get
\begin{equation*}\label{good-upper-bound}
\sum_{n=1}^\infty \sup_{k\in \N_0}\E\bigg(\frac{M^p(k; a^{n})-
|S(k;a^n)|^p}{f(a^{n})^p}\bigg)
\leq \frac{D_{a,p}}{1-c}\sum_{n=0}^\infty \sup_{k\in \N_0}
\E\Big|\frac{S(k;a^n)}{f(a^n)}\Big|^p
<\infty.
\end{equation*}
Consequently, 
for $l=0,1, \cdots, a-1$, we have almost surely
\begin{equation} \label{Eq:Mp}
\frac{M^p(la^n; a^{n})-|S(la^n;a^n)|^p}{f(a^{n})^p} \to 0,
\ \ \  \mathrm{as} \ n\to\infty.
\end{equation}
Note that \eqref{Eq:MomCon2}  implies that
\begin{equation*}
\frac{|S(la^n;a^n)|^p}{f(a^{n})^p}\to 0, \ \ \mathrm{ as } \  \ \
n\to\infty.
\end{equation*}
It follows from this and (\ref{Eq:Mp}) that almost surely
\begin{equation}\label{limit-l-zero}
\frac{M^p(la^n, a^{n})}{f(a^{n})^p}\to 0, \ \mathrm{as} \ n\to\infty.
\end{equation}
Since
\begin{equation}\label{phi-2}
\inf _n\frac{f(a^{n+1})}{f(a^{n})}\ge C_1^{\lfloor \log_2 a\rfloor} >1,
\end{equation}
we have
\begin{equation}\label{limit-zero}
\begin{split}
\frac{M(a^n; (a-1)a^{n})}{f(a^{n})}&=\frac{f(a^{n+1})-f(a^{n})}
{f(a^{n})}\frac{M(a^n; (a-1)a^{n})}{f(a^{n+1})-f(a^{n})}\\
&\geq
(C_1^{\lfloor \log_2 a\rfloor}-1)\frac{M(a^n; (a-1)a^{n})}{f(a^{n+1})-f(a^{n})}.
\end{split}
\end{equation}
Notice that
\[
\begin{split}
M(a^n; (a-1)a^n)\leq& M(a^n; a^n)+M(2a^n; a^n)+\cdots +M((a-1)a^n; a^n),
\end{split}
\]
we derive from \eqref{limit-l-zero} and \eqref{limit-zero} that almost surely
\begin{equation}\label{Eq:CLM}
\frac{M(a^n; (a-1)a^{n})}{f(a^{n+1})-f(a^{n})}\to 0, \ \mathrm{as} \ n\to\infty.
\end{equation}

Now  by the assumption on $f$,
\begin{equation*}\label{phi-3}
\frac{f(a^{n+1})}{f(a^{n})}\leq C_2^{\lfloor \log_2 a\rfloor}
 \end{equation*}
and by using Theorem 9.1 in Chobanyan, Levental and Mandrekar \cite{clm},
we see that (\ref{Eq:CLM}) implies
\begin{equation*}
\lim_{n \to \infty} \frac{S(0; n)} {f(n)} = 0\qquad \hbox{a.s.}
\end{equation*}
This finishes the proof of Theorem \ref{main-thm1} in the case $p>1$. The case $0<p\leq 1$
follows similarly and we omit the details.
\end{proof}

\section{Rectangular sums: the chaining method}
\label{Sec:Main2}


In this section we prove Theorem \ref{main-thm2}. We first state
a random field version of the  Toeplitz Lemma from M\'oricz
\cite{Moricz-78}. It will play a crucial role in the proof of Theorem
\ref{main-thm2}.
\begin{lemma}\label{Toeplitz-lemma}
Let $\{w(\m;\k), \m, \k \in \N_0^d\}$ be a set of non-negative numbers
with the following two properties: There is a finite constant $C$ such that
\begin{equation}\label{upper-bound-toeplitz}
\sup_{\m\in \N_0^d} \, \sum_{ \k \in \N_0^d}w(\m;\k)\leq C
\end{equation}
 and
\begin{equation}\label{limit-toeplitz}
\lim_{|\m|\to\infty}w(\m;\k)=0
\end{equation}
for all $\k\in \N_0^d$. If $\{s(\k):\  \k\in \N_0^d\}$ is a sequence of real numbers such that
\begin{equation}\label{term-to-zero-toeplitz}
s(\k)\to 0\ \ \mathrm{ as }\ \ |\k|\to \infty,
\end{equation}
then
\begin{equation*}
t(\m)=\sum_{\k\in \N_0^d} w(\m;\k)s(\k)\to 0\ \ \mathrm{ as }\ \  |\m|\to\infty.
\end{equation*}
\end{lemma}


For $s=1,2,\cdots, d$, define
\begin{equation} \label{Eq:Ms}
M_{s}(\m;\, \n) = \max_{1\leq k_1 \leq n_1, \cdots, 1\leq k_s\leq n_s}
|S(m_1,\cdots, m_d;\, k_1,\cdots, k_s, n_{s+1}, \cdots, n_d)|,
\end{equation}
and
\begin{equation} \label{s-s-1-maximum}
\begin{split}
F_{s,s-1}(\n)
&=\sup_{\m\in \N_0^d}\E\bigg(\frac{M_{s}^p(\m;a^{\n})-M^p_{s-1}(\m;a^{\n})}{\varphi_1(a^{n_1})^p
\, \cdots \varphi_d(a^{n_d})^p}\bigg).
\end{split}
\end{equation}
We call $M_{s}(\m;\, \n)$ the $s$-dimensional maximal sum of size $\n$, and $|S(\m;\n)|$ the $0$-dimensional maximal sum.
We will prove several maximal moment inequalities. The first is for the case of $0<p\leq 1$.


\begin{lemma}
Let $0 < p \leq1$ and let  $\varphi_s:\rr_+\to \rr_+$  ($s=1, \cdots, d$) be the functions  as in
Theorem \ref{main-thm2}.  There is  a constant $\kappa_1  \in (0,1)$ such that for $s=1,\,
\cdots,d$,
\begin{equation}  \label{p-less-than-one}
\begin{split}
\sum_{\n \in \N^d}  F_{s, s-1}(\n)
\leq  \frac{1}{1- \kappa_1}\sum_{\n \in \N_0}  \sup_{\m \in
\N_0^d}\E\bigg(\frac{M^p_{s-1}(\m;\, a^{\n})}{\varphi_1(a^{n_1})^p \cdots
\varphi_d(a^{n_d})^p}\bigg).
\end{split}
\end{equation}
\end{lemma}

\begin{proof}
We start by establishing a recursion for  $F_{s,s-1}(\n)$ in $\n$ for each $s=1, \cdots d$.
By  definition (\ref{Eq:Ms}) we see that for $l=2, \cdots, a$,
\begin{equation}\label{Eq:Mlp1}
\begin{split}
&M_{d}(\m;\, a^{n_1+1},\cdots, a^{n_{d-1}+1}, l a^{n_d})\\
&\leq   M_{d}(\m;\, a^{n_1+1}, \cdots, a^{n_{d-1}+1},(l-1)a^{n_d})\\
&\quad \quad  + M_d(m_1, \ldots,m_{d-1}, m_d+(l-1)a^{n_d};\, a^{n_1+1},
\cdots,a^{n_{d-1}+1}, a^{n_d}).
\end{split}
\end{equation}
It follows from \eqref{Eq:Mlp1} and the elementary inequality $|x+y|^p\leq |x|^p+|y|^p$  that
\[
\begin{split}
&M_{d}^p(\m;a^{n_1+1},a^{n_2+1}, \cdots ,la^{n_d}) \\
&\leq  M^p_{d}(\m;a^{n_1+1},a^{n_2+1}, \cdots, a^{n_{d-1}+1}, (l-1)a^{n_d})  \\
&+M^p_d(m_1,\cdots, m_{d-1},m_d+(l-1)a^{n_d};a^{n_1+1},a^{n_2+1},
\cdots, a^{n_{d-1}+1},a^{n_d}).
\end{split}
\]
By iterating the above inequality, we derive
\begin{equation}\label{Eq:Mpi}
\begin{split}
&M_{d}^p(\m;\, a^{n_1+1},a^{n_2+1}, \cdots, a^{n_d+1}) \\
&\leq  M^p_{d}(m_1,\,\cdots, m_d;a^{n_1+1},a^{n_2+1}, \cdots,
a^{n_{d-1}+1},a^{n_d})\\
&\ \ \ + M^p_d(m_1,\,\cdots,m_{d-1},
m_d+a^{n_d};a^{n_1+1},a^{n_2+1}, \cdots, a^{n_{d-1}+1}, a^{n_d})\\
&\ \ \ + M^p_d(m_1,\,\cdots,
m_{d-1},m_d+2a^{n_d};a^{n_1+1},a^{n_2+1}, \cdots, a^{n_{d-1}+1}, a^{n_d}) + \cdots \\
&\ \ \ + M^p_d(m_1, \,\cdots,m_{d-1},
m_d+(a-1)a^{n_d};a^{n_1+1},a^{n_2+1}, \cdots, a^{n_{d-1}+1}, a^{n_d}).\\
\end{split}
\end{equation}
By adding and subtracting terms in (\ref{Eq:Mpi}),  we get
\[
\begin{split}
&M_{d}^p(\m; \, a^{n_1+1}, \cdots, a^{n_d+1})
-M_{d-1}^p(\m; \, a^{n_1+1},\cdots, a^{n_d+1})\\
&\leq
-M_{d-1}^p(m_1,\cdots, m_d;\, a^{n_1+1},\cdots, a^{n_d+1})\\
&\quad + M^p_d(m_1,\cdots, m_d;\, a^{n_1+1}, \cdots,a^{n_{d-1}+1}, a^{n_d})\\
&\quad  - M_{d-1}^p(m_1,\cdots, m_d;a^{n_1+1},\cdots, a^{n_{d-1}+1},a^{n_d})\\
&\quad  + M_{d-1}^p(m_1,\cdots, m_d;a^{n_1+1},\cdots,a^{n_{d-1}+1}, a^{n_d})\\
&\quad  + M^p_d(m_1,\cdots,m_{d-1}, m_d+a^{n_d};\, a^{n_1+1}, \cdots, a^{n_{d-1}+1}, a^{n_d})\\
&\quad  - M^p_{d-1}(m_1,\cdots,m_{d-1}, m_d+a^{n_d};\, a^{n_1+1}, \cdots, a^{n_{d-1}+1}, a^{n_d})\\
&\quad  + M_{d-1}^p(m_1,\cdots, m_{d-1},m_d+a^{n_d};\, a^{n_1+1},\cdots, a^{n_{d-1}+1}, a^{n_d})\\
&\quad  + \cdots + \\
 &\quad  + M^p_d(m_1,\cdots,m_{d-1},
m_d+(a-1)a^{n_d};\, a^{n_1+1},\cdots, a^{n_{d-1}+1}, a^{n_d})\\
&\quad  - M_{d-1}^p(m_1,\cdots,m_{d-1}, m_d+(a-1)a^{n_d};\, a^{n_1+1},\cdots, a^{n_{d-1}+1}, a^{n_d})\\
&\quad  + M_{d-1}^p(m_1,\cdots,m_{d-1}, m_d+(a-1)a^{n_d};\, a^{n_1+1},\cdots, a^{n_{d-1}+1}, a^{n_d}).\\
\end{split}
\]
Dividing both sides by $\varphi_1(a^{n_1+1})^p \cdots \varphi_d(a^{n_d+1})^p$, taking
expectations and then the supremum over $\m \in \N_0$
we get,
\[
\begin{split}
F_{d, d-1}(n_1+1, \cdots, n_d+1)&\leq a \frac{\varphi_d(a^{n_d})^p}
{\varphi_d(a^{n_d+1})^p} F_{d, d-1}(n_1+1,\cdots , n_{d-1}+1,n_d) \\
&\ \ \ + H(n_{1}+1,\cdots , n_{d-1}+1, n_d),
\end{split}
\]
where
\[
\begin{split}
&H(n_{1}+1,\cdots, n_{d-1}+1, n_d)\\
&= \sup_{\m \in  \N_0^d}\bigg[ \frac{\varphi_d(a^{n_d})^p}{\varphi_d(a^{n_d+1})^p}
\sum_{l=0}^{a-1} \E\Big(\frac{M_{d-1}^p(m_1,\cdots,m_{d-1},
m_d+la^{n_d};a^{n_1+1},\cdots, a^{n_{d-1}+1}, a^{n_d})}{\varphi_1(a^{n_1+1})^p
\cdots \varphi_{d-1}(a^{n_{d-1}+1})^p \varphi_d(a^{n_d})^p}\Big)\\
& \qquad \qquad \quad - \frac{M_{d-1}^p(m_1,\cdots,
m_d;a^{n_1+1},\cdots a^{n_d+1})}{\varphi_1(a^{n_1+1})^p
 \cdots \varphi_d(a^{n_d+1})^p}\bigg]\\
&\leq   a \frac{\varphi_d(a^{n_d})^p}{\varphi_d(a^{n_d+1})^p}
\sup_{\m  \in \N_0^d}  \E\bigg(\frac{M_{d-1}^p(m_1,\cdots, m_d;a^{n_1+1},\cdots, a^{n_{d-1}+1}
a^{n_d})}{\varphi_1(a^{n_1+1})^p  \cdots \varphi_{d-1}(a^{n_{d-1}+1})^p
\varphi_d(a^{n_d})^p}\bigg).
\end{split}
\]
 The terms in $H(n_{1}+1,\cdots, n_{d-1}+1, n_d)$ help us get the recursion we need.
By the assumption on $\varphi_d$, we have
$$
a
\frac{\varphi_d(a^{n_d})^p}{\varphi_d(a^{n_d+1})^p} \leq
\frac{a}{C_{1,d}^{p\lfloor \log_2a\rfloor}}:=k_{d} <
1.$$
 Thus we get the recursion formula for $F_{d, d-1}$:
\[
\begin{split}
F_{d, d-1}(n_1+1, \cdots, n_d+1)&\leq k_d F_{d, d-1}(n_1+1,\cdots, n_{d-1} +1,n_d) \\
& \quad + \sup_{\m \in \N_0^d}\E\bigg(\frac{M_{d-1}^p(m_1,\cdots, m_d;\,a^{n_1+1},\cdots,
a^{n_d})}{\varphi_1(a^{n_1+1})^p   \cdots
\varphi_d(a^{n_d})^p}\bigg).
\end{split}
\]
Summing this recursion over $\n\in \N_0^d$, we get
\[
 \sum_{\n\in \N^d} F_{d, d-1}(\n)  \\
\leq  \frac{1}{1-k_d}\sum_{\n \in \N_0^d} \sup_{\m  \in
\N_0}\E\bigg(\frac{M^p_{d-1}(\m; \, a^{\n})}{\varphi_1(a^{n_1})^p \cdots
\varphi_d(a^{n_d})^p}\bigg).
\]
With essentially the same argument  as above we get for $s=1, \cdots , d-1$ that
\begin{equation}\label{p-less-than-one2}
\sum_{\n \in \N^d} F_{s, s-1}( \n)
\leq  \frac{1}{1-k_s}\sum_{\n \N_0^d}\sup_{ \m\in \N_0^d}
\E\bigg(\frac{M^p_{s-1}(\m;\, a^{\n})}{\varphi_1(a^{n_1})^p \cdots
\varphi_d(a^{n_d})^p}\bigg).
\end{equation}
Thus, we obtain (\ref{p-less-than-one})  with $\kappa_1 = \max_{1 \le s \le d} k_s$.
\end{proof}


The next lemma establishes  a maximal moment inequality for $p > 1$.

\begin{lemma}
Let $ p > 1$ and let   $\varphi_s:\rr_+\to \rr_+$  ($s=1,\cdots, d$) be
as in Theorem \ref{main-thm2}. There exist constants $\kappa_2  \in (0,1)$
and $D_{a,p} \in (0, \infty)$ such that for $s=1,\cdots,d$,
\begin{equation} \label{p-more-than-one}
\sum_{\n \in \N^d} F_{s, s-1}(\n)
\leq  \frac{D_{a,p}}{1- \kappa_2}\sum_{\n \in \N_0^d}
\sup_{\m \in \N_0^d} \E\bigg(\frac{M^p_{s-1}(\m;\, a^{\n})} {\varphi_1(a^{n_1})^p \cdots \varphi_d(a^{n_d})^p}\bigg).
\end{equation}
\end{lemma}

\begin{proof}

We start with establishing a recursion for  $F_{s,s-1}(\n)$ in $\n$ for each $s=1, \cdots,  d$.
Observe that for $l=2,\cdots,a$,
\begin{equation}\label{Eq:Mlp3}
\begin{split}
&M_{d}(\m;\, a^{n_1+1},\cdots, a^{n_{d-1}+1}, l a^{n_d})\\
&\leq \max \bigg\{M_{d}(\m;\, a^{n_1+1}, \cdots, a^{n_{d-1}+1}, (l-1)a^{n_d}); \\
&\qquad \qquad  \quad M_{d-1}(\m;\, a^{n_1+1}, \cdots, a^{n_{d-1}+1},(l-1)a^{n_d})\\
&\qquad \qquad\quad  + M_d(m_1, \ldots,m_{d-1}, m_d+(l-1)a^{n_d};\, a^{n_1+1},
\cdots,a^{n_{d-1}+1}, a^{n_d})\bigg\}.
\end{split}
\end{equation}
By using the elementary inequality $|x+y|^p\leq 2^{p-1}(|x|^p+|y|^p)$  we see that
 for $l=2,\cdots, a$,
\begin{equation}\label{Eq:Md}
\begin{split}
&M_{d}^p(\m;\, a^{n_1+1},\, \cdots, a^{n_{d-1}+1},l a^{n_d}) \\
& \leq (2^{p-1}-1)M^p_{d-1}(\m;\, a^{n_1+1},
\cdots, a^{n_{d-1}+1}, (l-1)a^{n_d}) \\
&  \qquad \quad + M^p_{d}(\m;\, a^{n_1+1},
\cdots, a^{n_{d-1}+1}, (l-1)a^{n_d}) \\
&  \qquad \quad + 2^{p-1}M^p_d(m_1, \,\cdots, m_{d-1},m_d
+(l-1)a^{n_d};\, a^{n_1+1}, \cdots, a^{n_{d-1}+1}, a^{n_d}).
\end{split}
\end{equation}
Define
\[
F_{d, d-1}(\n)
=\sup_{\m \in \N_0^d}\E\bigg( \frac{M_{d}^p(\m;\, a^{\n})-M^p_{d-1}(\m; \, a^{\n})} {\varphi_1(a^{n_1})^p\,\cdots \varphi_d(a^{n_d})^p}\bigg).
\]

In order to derive a recursion for $F_{d, d-1}(\n)$ we add and subtract terms in   \eqref{Eq:Md} to get
\begin{equation}\label{Eq:add-subtract}
\begin{split}
&M_{d}^p(\m;\, a^{n_1+1},\, \cdots, a^{n_{d-1}+1},l a^{n_d})-M_{d-1}^p(\m;\, a^{n_1+1},\, \cdots, a^{n_{d-1}+1}, l a^{n_d}) \\
& \leq -M_{d-1}^p(\m;\, a^{n_1+1},\, \cdots, a^{n_{d-1}+1},l a^{n_d})\\
&\quad +(2^{p-1}- 1)M^p_{d-1}(\m;\, a^{n_1+1},
\cdots, a^{n_{d-1}+1}, (l-1)a^{n_d}) \\
& \quad  + \big[M^p_{d}(\m;\, a^{n_1+1},
\cdots, a^{n_{d-1}+1}, (l-1)a^{n_d})\\
&\qquad \qquad \quad -M^p_{d-1}(\m;\, a^{n_1+1},
\cdots, a^{n_{d-1}+1}, (l-1)a^{n_d})\big]\\
&\quad   + 2^{p-1}\bigg[M^p_d(m_1,\cdots, m_{d-1},m_d
+(l-1)a^{n_d};\, a^{n_1+1}, \cdots, a^{n_{d-1}+1}, a^{n_d})\\
&\qquad\qquad \quad  -M^p_{d-1}(m_1,\cdots, m_{d-1},m_d
+(l-1)a^{n_d};\, a^{n_1+1}, \cdots, a^{n_{d-1}+1}, a^{n_d})\bigg]\\
&\quad + 2^{p-1}M^p_{d-1}(m_1,\cdots, m_{d-1},m_d
+(l-1)a^{n_d};\, a^{n_1+1}, \cdots, a^{n_{d-1}+1}, a^{n_d}).
\end{split}
\end{equation}

Inequality \eqref{Eq:add-subtract} is an analog of (1.7) in \cite{nane-xiao-zeleke}.
Then, by dividing both sides of \eqref{Eq:add-subtract} by
$$\varphi_1(a^{n_1+1})^p\,\cdots \varphi_{d-1}(a^{n_{d-1}+1})^p\varphi_d(la^{n_d})^p,$$
we arrive at a recursion similar to equation (1.9) in \cite{nane-xiao-zeleke}. Continuing as in the proof of Theorem 1.1
(the arguments  between Equations (1.7) and (1.10)) in \cite{nane-xiao-zeleke} and recalling the fact
$$
(1+(a-1)2^{p-1})\frac{\varphi_d(a^n)^p}{\varphi_d(a^{n+1})^p}\leq \frac{(1+(a-1)2^{p-1})}{C_{1,d}^{p\lfloor \log_2a\rfloor}}:= c_d<1,
$$
we obtain the following recursion
\[
\begin{split}
&F_{d, d-1}(n_1+1, \,\cdots ,n_d+1) \\
&\leq c_dF_{d, d-1}(n_1+1, n_2+1,\cdots , n_{d-1}+1, n_d)\\
&\quad +D_{a,p}\sup_{\m\in \N_0^d}\E\bigg(\frac{M^p_{d-1}(\m;a^{n_1+1},\, \cdots, a^{n_{d-1}+1}, a^{n_d})}
{\varphi_1(a^{n_1+1})^p \cdots \varphi_{d-1}(a^{n_{d-1}+1})^p\varphi_d(a^{n_d})^p}\bigg),
\end{split}
\]
where $D_{a,p}=2^{p-1}\big[(a-1)^{p-1}+(a-2)^{p-1}+\cdots+1+(a-1)\big]$.
By summing up this recursion over $\n$, we get
\begin{equation} \label{recursion-d-to-d-1}
\sum_{\n \in \N_0^d}   F_{d, d-1}(\n)
\leq  \frac{D_{a,p}}{1-c_d}\sum_{\n \in \N_0^d}
\sup_{\m\in \N_0^d}\E\bigg(\frac{M^p_{d-1}(\m;a^{n_1},  \cdots, a^{n_d})}
{\varphi_1(a^{n_1})^p \cdots \varphi_d(a^{n_d})^p}\bigg).
\end{equation}

In the same vein we derive a relation for $F_{s,s-1}(\n)$ similar to the inequality \eqref{recursion-d-to-d-1}.
For $l=2,\cdots,a$, we observe first that
 \begin{equation}\label{Eq:s}
 \begin{split}
&M_{s}(\m;a^{n_1+1}, \cdots, a^{n_{s-1}+1}, la^{n_s},a^{n_{s+1}+1},\cdots, a^{n_d+1} )\\
&\leq \max \bigg\{M_{s}(\m;a^{n_1+1}, \cdots , a^{n_{s-1}+1}, (l-1)a^{n_s},a^{n_{s+1}+1},\cdots, a^{n_d+1}),\\
&\qquad \quad M_{s-1}(\m;a^{n_1+1}, \cdots , a^{n_{s-1}+1},(l-1)a^{n_s},a^{n_{s+1}+1},\cdots, a^{n_d+1})\\
&\qquad \quad + M_s(m_1,\cdots , m_{s-1}, m_s+(l-1)a^{n_s}, m_{s+1},\cdots, m_d;\\
& \ \ \ \ \ \ \ \ \ \ \ \ \ \ \ \ \ \ \ \ \ \ \ \ \ \ a^{n_1+1}, \cdots, a^{n_{s-1}+1}, a^{n_s},a^{n_{s+1}+1},\cdots, a^{n_d+1})\bigg\}.
\end{split}
\end{equation}
It follows that for $l=2,\cdots,a$,
\[
\begin{split}
&M_{s}^p(\m;a^{n_1+1}, \cdots, a^{n_{s-1}+1},  la^{n_s},a^{n_{s+1}+1},\cdots, a^{n_d+1}) \\
&\leq  (2^{p-1}-1)M^p_{s-1}(\m;a^{n_1+1}, \cdots, (l-1)a^{n_s},a^{n_{s+1}+1},\cdots, a^{n_d+1})  \\
&\qquad + M^p_{s}(\m;a^{n_1+1}, , \cdots, a^{n_{s-1}+1} , (l-1)a^{n_s},a^{n_{s+1}+1},\cdots, a^{n_d+1})\nonumber\\
&\qquad + 2^{p-1}M^p_s(m_1,\cdots , m_{s-1}, m_s+(l-1)a^{n_s}, m_{s+1},\cdots, m_d; \\
           &\ \ \ \ \ \ \ \ \ \ \ \ \ \ \ \ \ \ \ \ \ \ \ \ \ \ \ \ a^{n_1+1}, \cdots, a^{n_{s-1}+1},  a^{n_s},a^{n_{s+1}+1},\cdots, a^{n_d+1}).
\end{split}
\]

By the assumption on $\varphi_s$ we have
$$
(1+(a-1)2^{p-1})\frac{\varphi_s(a^n)^p}{\varphi_s(a^{n+1})^p}\leq \frac{(1+(a-1)2^{p-1})}
{C_{1,s}^{p\lfloor \log_2a\rfloor}}:= c_s<1.
$$
By this fact, using similar arguments for deriving \eqref{Eq:add-subtract} and \eqref{recursion-d-to-d-1} 
we get the result of the lemma for $s=1, \cdots, d$.
\end{proof}

We are ready to prove Theorem \ref{main-thm2}.

\begin{proof}[{Proof of Theorem \ref{main-thm2}}]
We use the recursions obtained in \eqref{p-less-than-one} and \eqref{p-more-than-one}, together with Lemma  \ref{Toeplitz-lemma},
to prove the conclusion of the theorem.

\noindent {\bf Step 1:} We first show that the maximal sums  of dimension $s=1,2,\cdots, d$ of size $a^\n$
defined in (\ref{Eq:Ms}) converge to zero almost surely.  Recall that
$$
\big|S(\m; a^{n_1},\, \cdots, a^{n_d}) \big|= M_0(\m;a^{n_1},  \cdots ,  a^{n_d}).
$$
By assumption (\ref{Eq:MomCon1}), we have
$$
\sum_{\n \in \N_0^d }
\sup_{\m \in \N_0^d}\E\bigg(\frac{M_0^p(\m; a^{n_1},  \cdots,  a^{n_d})}
{\varphi_1(a^{n_1})^p \cdots \varphi_d(a^{n_d})^p}\bigg)<\infty.
$$
This,  \eqref{p-less-than-one} and \eqref{p-more-than-one} for $s=1$ imply that
$$
\sum_{\n \in \N_0^d}  \sup_{\m \in \N_0^d}\E\bigg(\frac{M^p_{1}(\m ; a^{n_1},
\cdots,  a^{n_d})}{\varphi_1(a^{n_1})^p \cdots
\varphi_d(a^{n_d})^p}\bigg)<\infty.
$$
By this last result and  Equations \eqref{p-less-than-one} and \eqref{p-more-than-one} for $s=2$, we obtain  
$$
\sum_{\n \in \N_0^d } \sup_{\m \in \N_0^d }\E\bigg(\frac{M^p_{2}(\m; a^{n_1},
\cdots,  a^{n_d})}{\varphi_1(a^{n_1})^p \cdots
\varphi_d(a^{n_d})^p}\bigg)<\infty.
$$

Continuing the same way in the order $s=3,   \cdots, d$ and using  \eqref{p-less-than-one} and
\eqref{p-more-than-one}  yields
\begin{equation}\label{all-finite-bound}
\sum_{\n \in \N_0^d}  \sup_{\m \in
\N_0^d}\E\bigg(\frac{M^p_{s}(\m; a^{n_1},
\cdots, a^{n_d})}{\varphi_1(a^{n_1})^p \cdots
\varphi_d(a^{n_d})^p}\bigg)<\infty
\end{equation}
 for every $s=1,\,\cdots, d$.

It follows from (\ref{all-finite-bound}) that for every $s=1,\,\cdots, d$. and $l_1, l_2,\cdots, l_d \in
\{0,1,\, \cdots, a-1\}$ almost surely
\begin{equation}\label{md-lim}
\frac{M_{s}(l_1a^{n_1},  \cdots, l_da^{n_d};a^{n_1},  \cdots, a^{n_d})}{\varphi_1(a^{n_1})
 \cdots \varphi_d(a^{n_d})}\to 0,\ \  \mathrm{as}\ |\n|\to \infty.
\end{equation}


\noindent {\bf Step 2:} We identify the terms to be used in   Lemma  \ref{Toeplitz-lemma}.

For any $k_1,\cdots, k_d\geq 0$ define
\begin{equation}\label{Eq:ak}
A(k_1,\cdots, k_d)=\sum_{j=1}^d\sum_{l_j=0}^{a-1}M_d\big(l_1a^{k_1},  \cdots, l_da^{k_d};a^{k_1},  \cdots , a^{k_d}\big).
\end{equation}
If $a^{n_1}\leq m_1<a^{n_1+1},  \, \cdots, a^{n_d}\leq m_d<a^{n_d+1}$ we have

\[
\begin{split}
\frac{|S(\0;\, \m )|}{\varphi_1(m_1)\cdots \varphi_d(m_d)}
&\leq  \sum_{\k \leq \n }\frac{A(k_1,\cdots, k_d)}{\varphi_1(a^{n_1})\cdots \varphi_d(a^{n_d})}\\
&\leq \frac{\varphi_1(a^{n_1+1}) \cdots \varphi_d(a^{n_d+1})}{\varphi_1(a^{n_1})\cdots \varphi_d(a^{n_d})}\sum_{\k \leq \n } w(\n; \k)\\
&\qquad \times\frac{A(k_1,\cdots, k_d)}{(\varphi_1(a^{k_1+1})-\varphi_1(a^{k_1}))\cdots (\varphi_d(a^{k_d+1})-\varphi_d(a^{k_d}))}\\
&=\frac{\varphi_1(a^{n_1+1})\cdots \varphi_d(a^{n_d+1})}{\varphi_1(a^{n_1})\cdots \varphi_d(a^{n_d})}\sum_{\k \leq \n}w(\n;\k)s(\k),
\end{split}
\]
where
\begin{equation}
\begin{split}
w(\n;\k)=
\frac{(\varphi_1(a^{k_1+1})-\varphi_1(a^{k_1})) \cdots (\varphi_d(a^{k_d+1})-\varphi_d(a^{k_d}))}
{\varphi_1(a^{n_1+1})\varphi_2(a^{n_2+1})\cdots \varphi_d(a^{n_d+1})}
\end{split}
\end{equation}
if $k_1\leq n_1,\cdots, k_d
\leq n_d$, and is defined to be $0$ otherwise, and where
$$
s(\k)=\frac{A(k_1,\cdots, k_d)}{(\varphi_1(a^{k_1+1})-\varphi_1(a^{k_1})) \cdots (\varphi_d(a^{k_d+1})-\varphi_d(a^{k_d}))}.
$$

Therefore, we need to verify the conditions of Lemma \ref{Toeplitz-lemma} for this $w(\n;\k)$ and  $s(\k)$.

\begin{enumerate}
\item It can be verified that
$$
\sum_{\k \geq {\bf 1}}w(\n;\k) \leq 1
$$
for all $\n \in \N_0^d$ and

\item $$
w(\n;\k) \to 0 \ \ \mathrm{ as }\ \ |\n| \to \infty
$$
for all $\k \in \N_0^d$.

\item Next we show that
$$
s(\k)\to 0 \ \mathrm{as} \ |\k|\to\infty.
$$

\end{enumerate}
It follows from \eqref{md-lim} and (\ref{Eq:ak}) that
\[
\begin{split}
\frac{A(k_1,\cdots, k_d)}{\varphi_1(a^{k_1}) \cdots \varphi_d(a^{k_d})}\to 0\ \ \mathrm{ as }\ |\k| \to \infty.
\end{split}
\]
By using the fact that for $s=1,\cdots, d$
\begin{equation}\label{phi-d}
\inf _n\frac{\varphi_s(a^{n+1})}{\varphi_s(a^{n})}\ge C_{3}^{\lfloor \log_2 a\rfloor} >1
\end{equation}
we get
\begin{equation}
\begin{split}
&\frac{A(k_1,\cdots, k_d)}{\varphi_1(a^{k_1}) \cdots \varphi_d(a^{k_d})}\\
&=\frac{(\varphi_1(a^{k_1+1})-\varphi_1(a^{k_1})) \cdots (\varphi_d(a^{k_d+1})-\varphi_d(a^{k_d}))}{\varphi_1(a^{k_1}) \cdots \varphi_d(a^{k_d})}\\
&\times\frac{A(k_1, \cdots, k_d)}{(\varphi_1(a^{k_1+1})-\varphi_1(a^{k_1})) \cdots (\varphi_d(a^{k_d+1})-\varphi_d(a^{k_d}))}\\
&\geq \big(C_{3}^{\lfloor \log_2 a\rfloor}-1\big)^d s(\k).
\end{split}
\end{equation}
Hence almost surely $s(\k) \to 0 \ \mathrm{as}\ |\k|\to \infty.$ Therefore we have verified all the conditions in Lemma \ref{Toeplitz-lemma}.

\noindent {\bf Step 3:} We combine the results in Steps 1 and 2 to finish the proof.
To this end, we notice that, by assumption 
\begin{equation*}\label{phi-3b}
\frac{\varphi_s(a^{n+1})}{\varphi_s(a^{n})}\leq C_{4}^{\lfloor \log_2 a\rfloor}
\end{equation*}
for $s=1,\, \cdots, d$. It follows from Lemma \ref{Toeplitz-lemma} that
\[
\begin{split}
&\frac{|S(\0;\, \m )|}{\varphi_1(m_1) \cdots \varphi_d(m_d)}\\
&\leq \frac{\varphi_1(a^{n_1+1}) \cdots \varphi_d(a^{n_d+1})}
{\varphi_1(a^{n_1}) \cdots \varphi_d(a^{n_d})}\sum_{\k\leq \n}w(\n; \k)
s(\k)\\
&\leq  C_{4}^{\lfloor \log_2 a\rfloor d}\, \sum_{\k\leq \n}w(\n; \k)
s(\k)\to 0\ \ \ \ \mathrm{ as } \ |\n| \to\infty.
\end{split}
\]
Therefore, we have
$$
\frac{S(\0; \m) }{\varphi_1(m_1) \cdots
\varphi_d(m_d)}\to 0\ \ \mathrm{ as }\ \ |\m|\to \infty.
$$
This finishes the proof of Theorem \ref{main-thm2}.
\end{proof}

The previous proof  can be extended  to give a proof of the following result. Notice that it is slightly
more general than Theorem \ref{main-thm2}
because condition (\ref{Eq:MomCon1-2}) is more flexible than condition \eqref{Eq:MomCon1}
\begin{theorem}\label{Th:SLLN-2}
Let  $\varphi_s:\rr_+\to \rr_+$ ($s=1,\,\cdots, d$) be functions as in Theorem  \ref{main-thm2}.
Assume that $a_s\geq 2$ ($s=1,\,\cdots, d$) are integers that satisfy
$C_{3}^{\lfloor \log_2 a_s \rfloor}>\max\{a_s, a_s2^{p-1}\}$. If we have
\begin{equation}\label{Eq:MomCon1-2}
\sum_{\n \in \N_0^d}
\sup_{\m \in \N_0^d}\E\bigg(\frac{|S(\m;a_1^{n_1}, \, \cdots, a_d^{n_d})|^p}
{\varphi_1(a_1^{n_1})^p \cdots \varphi_d(a_d^{n_d})^p}\bigg)<\infty,
\end{equation}
then $\{\xi(\n), \n \in \N_0^d\}$ satisfies the SLLN (\ref{Def:SLLN}) with $ \Gamma_\n
= [{\bf 1},\, \n]$ and $\varphi(\n) = \varphi_1(n_1)\cdots \varphi_d(n_d)$.
\end{theorem}

\section{Applications}
\label{Sec:Appl}

In this section we show that Theorems \ref{main-thm1}  and \ref{main-thm2} can be applied to
various discrete-time or continuous-time random fields.

\subsection{Quasi-stationary random fields}
Let $f:  \N_0^d \to \R_+$ be a non-negative function. Recall from Moricz \cite{Moricz-80}
that  a real-valued random field $\{\xi_{\n}, \n \in \N_0^d\}$ is \emph{$f$-quasi-stationary},
if $\E \big(\xi_{\n}^2 \big) = 1$ for all $\n\in \N_0^d$, and
$$
\big|\E (\xi_{\m} \xi_{\m +\n}) \big|\leq f(\n), \ \ \hbox{ for all} \ \m, \n \in \N_0^d,\  (\hbox{so we can take }\, f(\0) = 1).
$$

The following is a corollary of Theorem \ref{main-thm2} with $p = 2$. For simplicity,
we only consider the case $d = 2$.
\begin{corollary}
\label{corollary f-qstationary}
Let  $\{\xi_{n,m}, \, n,m\in \N_0\}$ be an $f$-quasi-stationary
sequence and $\varphi_{1}, \varphi_{2}$ be non-decreasing
functions as in Theorem \ref{main-thm2} such that
 $D:=\displaystyle \sum_{m=0}^{\infty}\sum_{n=0}^{\infty} \frac{a^n\;a^m}
 {\varphi_{1}(a^n)^2 \varphi_{2}(a^m)^2} <
 \infty$.
Define
$$h(i,j)\equiv \sum_{m=[\log _a j]}^{\infty}\sum_{n=[\log_a i]}^{\infty} \frac{a^n
a^m}{\varphi_{1}(a^n)^2 \varphi_{2}(a^m)^2}.
$$
If
\begin{equation} \label{Eq:fh-con}
 \sum_{j=1}^{\infty}\sum_{i=1}^{\infty} f(i,j)h(i,j)<\infty,
\end{equation}
then $\{\xi_{n,m}, n,m\in \N_0\}$ satisfies the SLLN on rectangles with
$\varphi(n,m) = \varphi_1(n)\varphi_2(m).$
\end{corollary}


\begin{proof}
For any $k,l,n,m \in \N_0$
\begin{equation}\label{Eq:3d}
\begin{split}
\E\left|\frac{S(k,l;\, a^n,a^m)}{\varphi_{1}(a^n)\varphi_{2}(a^m)}\right|^2 &\leq
\sum_{j=0}^{a^m}\sum_{i=0}^{a^n}\frac{f(i,j)(a^n-i)(a^m-j)}{\varphi_{1}(a^n)^2\varphi_{2}(a^m)^2}\\
&\leq
\frac{a^n\;a^m}{\varphi_{1}(a^n)^2\varphi_{2}(a^m)^2}
\sum_{j=0}^{a^m}\sum_{i=0}^{a^n}f(i,j).
\end{split}
\end{equation}
This implies
\begin{equation}\label{Eq:3e}
\begin{split}
&\sum_{m=0}^{\infty}\sum_{n=0}^{\infty}\sup_{k,l\in \N_0} \E \left(
\frac{S(k,l;\, a^{n},a^{m})}{\varphi_{1}(a^n)\varphi_{2}(a^m)}\right)^2
\leq  \sum_{m=0}^{\infty}\sum_{n=0}^{\infty} \sum_{j=0}^{a^m}
\sum_{i=0}^{a^n}\frac{a^n\;a^m f(i,j)}{\varphi_{1}(a^n)^2\varphi_{2}(a^m)^2}\\
&\leq  D f(0,0) +\sum_{j=1}^{\infty}\sum_{i=1}^{\infty} f(i,j)
\sum_{m=[\log _a j]}^{\infty}\sum_{n=[\log_a i]}^{\infty}\frac{a^n a^m}
{\varphi_{1}(a^n)^2\varphi_{2}(a^m)^2}\\
&\leq  D f(0,0)+\sum_{j=1}^{\infty}\sum_{i=1}^{\infty}
f(i,j)h(i,j)<\infty.
\end{split}
\end{equation}
Hence the conclusion follows from Theorem \ref{main-thm2}.
\end{proof}

By using Corollary \ref{corollary f-qstationary}, we can improve a result of M\'oricz \cite{Moricz-89} on SLLN
of a class of quasi-orthogonal random fields.
Recall from\cite{Moricz-89} that a random field $\{X_{i,k}\}$ is called quasi-orthogonal if
$\E(X_{i,k}^2) = \sigma_{i,k}^2 < \infty $ for all $i,k \geq 1$ and there exists a double sequence
$\{\rho(m,n): \;m,n \geq 0 \}$ of nonnegative numbers such that
$$
| \E(X_{i,k}X_{j,l})| \leq \rho(|i-j|,|k-l|) \sigma_{i,k} \sigma_{j,l},\ \  \ ( \forall \ i,j,k,l\, \geq 1)
$$
and
\begin{equation}
\label{con:rho}
\sum_{m=0}^{\infty} \sum_{n=0}^{\infty} \rho(m,n) < \infty.
\end{equation}

Suppose $\{\lambda_{j}(m), m \in \N_0\}$ ($j = 1, 2$) are nondecreasing sequences of
positive  numbers  such that $ \displaystyle \liminf_{m \rightarrow
\infty}\frac{\lambda_{j}(2m)}{\lambda_{j}(m)}> 1, \; \mbox{for}\;j=1, 2\;\mbox{ and let }\;
\{ X_{ik}\}\;\mbox{ be a quasi-orthogonal random field}$.  M\'oricz \cite[Theorem 1 ]{Moricz-89}
showed that if
\begin{equation}
\label{quasi-orthogonal}
\sum_{i=1}^{\infty}\sum_{k=1}^{\infty}\frac{\sigma_{ik}^2}{\lambda_{1}^2(i)\lambda_{2}^2(k)}
(\log(i+1))^2(\log(k+1))^2 < \infty,
\end{equation}
then
\begin{equation}\label{Eq:M89}
\displaystyle \lim_{m +n \rightarrow
  \infty}\frac{1}{\lambda_{1}(m)\lambda_{2}(n)}\displaystyle
  \sum_{i=1}^{m}\sum_{k=1}^{n}X_{ik}=0\;\ \ \ \mbox{ a.s.}
\end{equation}
If we further assume that  $\sigma_{i,j}^2=1$  for all $i,k \geq 1$, then $\{X_{i,k}\}$ is
$\rho$-quasi-stationary. By taking $\varphi_j = \lambda_j$ ($j = 1, 2$) in Corollary \ref{corollary f-qstationary},
and noticing that \eqref{Eq:fh-con} is satisfied due to (\ref{con:rho}), and $\displaystyle
\sum_{m=0}^{\infty} \frac{a^n} {\varphi_{j}(a^n)^2} <
 \infty$ if and only if $\displaystyle \sum_{m=0}^{\infty} \frac{1}
 {\varphi_{j}(n)^2 } <  \infty$, we conclude that (\ref{Eq:M89}) holds provided
 $ \displaystyle  \sum_{k=1}^{\infty}\frac{1}{\lambda_{j}^2(k) } < \infty $ for  $ j = 1, 2.$




\subsection{Orthogonal random fields}

Following Klesov \cite{K82,Klesov03}, a random field $\{\xi(\n), \n \in \N_0^d\}$ is
said to be \emph{orthogonal} if
\[
\E(\xi(\n)) = 0, \ \E\big[\xi(\n)^2\big] < \infty \ \ \hbox{ and }\  \
\E\big[\xi(\m)\xi(\n)\big] =0 \ \hbox{ if } \m \ne \n.
\]
Klesov \cite{K82,Klesov03} proved that, if
\begin{equation}\label{Eq:Klesov}
\sum_{\n \in \N^d} \E(\xi(\n)^2)\, \prod_{i=1}^d
\bigg(\frac{\log (1 + n_i)} {n_i}\bigg)^2 < \infty,
\end{equation}
then
\begin{equation}\label{Eq:Klesov2}
\lim_{|\n| \to \infty} \frac{S(\0, \n)} {\prod_{i=1}^d n_i} = 0 \qquad \hbox{ a.s.}
\end{equation}

As a consequence of Corollary  \ref{Co:2}, we show that, under an extra condition (\ref{Eq:Klesov3})
which is weaker than quasi-stationarity, condition (\ref{Eq:Klesov}) can be weakened.

\begin{corollary}
Let $\{\xi(\n), \n \in \N_0^d\}$ be an orthogonal random field. If there exists a
finite constant $C_5$ such that
\begin{equation}\label{Eq:Klesov3}
\sup_{\m \in \N_0^d} \E\big(S(\m, 2^{\n})^2\big) \le C_5\, \E\big(S(\0, 2^{\n})^2\big)
\end{equation}
for all $\n \in \N^d$ and
\begin{equation}\label{Eq:Klesov4}
\sum_{\n \in \N^d} \frac {\E(\xi(\n)^2)} {\prod_{i=1}^d
n_i^2} < \infty,
\end{equation}
then (\ref{Eq:Klesov2}) holds.
\end{corollary}

\begin{proof} We take $p = 2$,  $g(\n) = C_5\, \E\big(S(\0, 2^{\n})^2\big)$,
and  $\varphi(\n) = n_1 \cdots n_d$ in Corollary  \ref{Co:2}. Then by the orthogonality we derive
\begin{equation}\label{Eq:Klesov5}
\begin{split}
\sum_{\n \in \N^d}  \frac{g( 2^{\n})} {\varphi(2^{\n})^{2}}
&= \sum_{\n \in \N^d}  \frac{\sum_{\k \le 2^{\n}} \E\big(\xi(\k)^2\big)} {\prod_{i=1}^d
2^{2n_i}}\\
&= \sum_{\k \in \N^d} \E\big(\xi(\k)^2\big) \sum_{\n \ge \lfloor \log_2 {\k} \rfloor} \frac{1} {\prod_{i=1}^d
2^{2n_i}} \\
&\le C\, \sum_{\k \in \N^d} \frac{\E\big(\xi(\k)^2\big)} {\prod_{i=1}^d
k_i^2}   < \infty.
\end{split}
\end{equation}
In the above, $ \lfloor \log_2 {\k} \rfloor = ( \lfloor \log_2 {k_1} \rfloor, \ldots,
\lfloor \log_2 {k_d} \rfloor)$ and the last inequality follows from (\ref{Eq:Klesov4}).
Hence the conclusion follows from Corollary  \ref{Co:2}.
\end{proof}

\subsection{Fractional stable random fields}
In recent years, several authors have studied local and asymptotic 
properties of isotropic and anisotropic fractional stable random fields; 
see for example, Bierm\'e and Lacaux \cite{bierme2009, bierme2015},
Panigrahi, Roy and Xiao \cite{PRX17}, Roy and Samorodnitsky \cite{RoySa08}, 
and Xiao \cite{Xiao10, Xiao08}. We refer to the books of Samorodnitsky 
and Taqqu \cite{ST94}, Cohen and Istas \cite{CI13}, Pipiras and Taqqu \cite{PT17}
for systematic and historical accounts on stable distributions and stable 
random fields.

In this section, we study SLLN for linear fractional stable sheets 
considered in Ayache, Roueff and Xiao \cite{ARX08}. This was the main 
motivation that initiated this paper.

For any given $0 < \alpha < 2$ and $H = (H_1,$ $ \ldots, H_d)\in (0,
1)^d$, we define an $\a$-stable random field $Z^{H} = \{Z^{
H}(\t), \t \in \R_+^d\}$ with values in $\R$ by
\begin{equation}\label{Eq:Rep-LFSS}
Z^{H}(\t) = \int_{\R^d} g_{_{H}}(\t, \s)\, M_\a(d\s),
\end{equation}
where $M_\a$ is a symmetric $\alpha$-stable random measure on $\R^d$ with
Lebesgue control measure and
\begin{equation}\label{Eq:functh}
\begin{split}
g_{_H} (\t, \s) &= \kappa \prod_{\ell=1}^d
\Big\{\big((t_\ell-s_\ell)_+ \big)^{H_\ell-1/\a} -
\big((-s_\ell)_+\big)^{H_\ell-1/\a} \Big\}.
\end{split}
\end{equation}
In the above $t_{+} = \max\{t, 0\}$, and $\kappa > 0$ is a normalizing constant such that
the scale parameter of $Z^{H}(\t)$ equals $\|\t\|^H$. $Z^{H}= \{Z^{H}(\t), \t \in \R_+^d\}$ is
called an $(d,1, \alpha)$-linear fractional stable sheet. Using (\ref{Eq:Rep-LFSS}) one
can verify that it has the following operator-scaling
property: For any $d\times d$ diagonal matrix $E = (b_{ij})$ with
$b_{ii} = b_i > 0$ for all $1 \le i \le d$ and $b_{ij}=0$ if $i\ne j$, we have
\begin{equation}\label{Eq:OSS}
\big\{  Z^{H}(E \t),\, \t \in \R^d \big\} \stackrel{d}{=} \bigg\{
\bigg(\prod_{j=1}^N b_j^{H_j}\bigg)\,  Z^{H}(\t),\ \t \in \R^d \bigg\}.
\end{equation}
Moreover, along  each direction of $\R^d_+$, ${Z}^{H}$ becomes a
(rescaled) real-valued linear fractional stable motion. See  Samorodnitsky and Taqqu \cite{ST94} 
or Pipiras and Taqqu \cite{PT17} for more information on the
later and other self-similar stable processes.

Ayache, Roueff and Xiao \cite{ARX08} have studied asymptotic properties,
modulus of continuity, fractal dimensions and local times of linear fractional
stable sheets. The following SLLN in equation \eqref{Eq:LFSS} was proved in Theorem 1 in \cite{ARX08}
by using the wavelet method. Here we show that it can also be obtained by
using Theorem \ref{main-thm2}.

We start with the following moment and tail inequalities for the maximal increments
of $Z^{H}$. These inequalities are of interest in their own and will have applications
beyond the scope of this paper, see Panigrahi, Roy and Xiao \cite{PRX17} and Xiao
\cite{Xiao10} for some related results.

\begin{lemma}\label{Lem:stablesuptail}
Let ${Z}^{H}$ be a $(d, 1, \alpha)$-LFSS as above.  Assume that
$1 < \alpha < 2$ and  $ 1/\alpha < H_j < 1$ for every $j = 1, \ldots, d$. Then for
any constants $\eps > 0$ and $0 < \gamma < \alpha$, there exists a positive
and finite constant $C_6$ such that for all intervals $T = [{\bf a}, {\bf b}] \subseteq
[\eps,\infty)^d$,
\begin{equation}\label{Eq:Moricz2}
\E\bigg(\sup_{t \in T} \big|Z^H(\t) - Z^H({\bf a})\big|^\gamma \bigg) \le C_6\,\bigg(\prod_{j=1}^d (b_j - a_j)^{H_j\gamma}
+ \sum_{k=1}^d a_k^{H_k \gamma} \prod_{j \ne k} (b_j - a_j)^{H_j \gamma}\bigg).
\end{equation}
Hence, for any $u > 0$,
\begin{equation}\label{Eq:maxtail}
\begin{split}
&\P\bigg(\sup_{\t \in T} \big|Z ^H(\t) - Z^H({\bf a})\big| \geq u\bigg)\\
& \leq
C_6 u^{-\gamma}  \bigg(\prod_{j=1}^d (b_j - a_j)^{H_j\gamma}
+ \sum_{k=1}^d a_k^{H_k \gamma} \prod_{j \ne k} (b_j - a_j)^{H_j \gamma}\bigg).
\end{split}
\end{equation}
\end{lemma}

\begin{proof}
Notice that (\ref{Eq:maxtail}) follows from \eqref{Eq:Moricz2} and the Markov inequality.
Moreover, for  $d = 1$, it is well known that \eqref{Eq:Moricz2}
holds (see \cite{Taka} or (2.8) in \cite{nane-xiao-zeleke}).
Hence, it suffices prove  \eqref{Eq:Moricz2} for $d \ge 2$.

The idea is to write  the increment $Z^H (\t) - Z ^H ({\bf a})$
as a sum of increments of (scaled) $(d-1, 1, \alpha)$-LFSS over intervals in the $d-1$-dimensional hyperplane
and to use the general bound for the maximal moment for the increments of LFSS over intervals established in \cite{ARX08}.
The difference is that $T$ was assumed to be in $[\eps,1]^d$ in \cite{ARX08}, while we allow $T$ to be anywhere
in $[\eps,\infty)^d$. So some careful modifications will have to be made.

Recall that, for any $\t \in T$, the increment of $Z^H$ over the rectangle $[{\bf a},\t]$, denoted by
$Z^H([{\bf a},\t])$, is defined as
\begin{equation}\label{increment-process}
Z^H([{\bf a},\t])=\sum_{{ \delta} \in \{0,1\}^d} (-1)^{d-\sum_i \delta_i}Z^H(\langle a_j+\delta_j(t_j-a_j)\rangle ).
\end{equation}
This corresponds to the measure of the set $[{\bf a},\t]$ by interpreting $Z^H$ as a signed measure defined
by $Z^H([\0,\t])=Z^H(\t)$ for all $\t\in \R_+^d$. (convention: $[0,t_j]:=[t_j,0]$ if $t_j<0$).

The proof of Lemma 20 in \cite[(5.21)]{ARX08} shows that for any $0 <\gamma < \alpha$,
\begin{equation}\label{Eq:Maximom}
\E \bigg(\sup_{\t \in T} \big|Z ^H \big([{\bf a},\t] \big)\big|^\gamma \bigg)
\le C\, \prod_{j=1}^d (b_j - a_j)^{H_j\gamma},
\end{equation}
where $C$ is a positive and finite constant that is independent of ${\bf a}$ and ${\bf b}$.

Observe that, in \eqref{increment-process},  the term corresponding to
$\delta=\langle 1 \rangle$ in the sum is $Z^H(\t)$, and since
$\sum_{\delta\in \{0,1\}^d} (-1)^{d-\sum_i \delta_i}=0$, we can write \eqref{increment-process} as
\begin{equation}\label{increment-1}
\begin{split}
Z^H([{\bf a},\t])&= \sum_{\delta\in \{0,1\}^d} (-1)^{d-\sum_i \delta_i}
\Big(Z^H(\langle a_j+\delta_j(t_j-a_j)\rangle)-Z^H({\bf a})\Big)\\
&= Z^H(\t)-Z^H({\bf a})+ \sum_{\delta\in \{0,1\}^d \setminus \langle 1\rangle}
(-1)^{d-\sum_i \delta_i}\Big(Z^H(\langle a_j+\delta_j (t_j-a_j)\rangle)-Z^H({\bf a})\Big).
\end{split}
\end{equation}
For every $\delta\in \{0,1\}^d\setminus \langle 1\rangle$, there is some
$k\in \{1,\,\ldots, d\}$ such that $\delta_k=0$. For any $\t \in \rr^{d}$, we denote by
$\widehat{\t}_k$ the vector $(t_1, \ldots, t_{k-1}, t_{k+1}, \ldots, t_d) \in \rr ^{d-1}$, and by
$(a_k, \widehat{\t}_k) $ the vector $(t_1, \ldots, t_{k-1}, a_k, t_{k+1}, \ldots, t_d) \in \rr ^{d}$.

The last sum in \eqref{increment-1} can be written as a sum of $d$ increments of
scaled $(d-1, 1, \alpha)$-linear fractional stable sheets. More precisely,
\begin{equation}\label{increment-2}
\begin{split}
&\sum_{\delta\in \{0,1\}^d \setminus \langle 1\rangle}
(-1)^{d-\sum_i \delta_i}\Big(Z^H(\langle a_j+\delta_j (t_j-a_j)\rangle)-Z^H({\bf a})\Big)\\
&= \sum_{k=1}^d \sum_{\stackrel{\delta \in \{0,1\}^{d-1} \setminus \langle 1\rangle}{\tiny{
\delta_k = 0}}}
(-1)^{d - \sum_{i} \delta_i}\Big(Z^H(\langle a_j+\delta_j (t_j-a_j)\rangle)-Z^H({\bf a})\Big)
\end{split}
\end{equation}
Notice that for each fixed $k$, the inner summation in the last line is the increment of
the process $- Z^H(a_k,\widehat{\t}_k)$ over the interval $[\widehat{\bf a}_k, \, \widehat{\bf b}_k]$.
By using the representation \eqref{Eq:Rep-LFSS}, we can verify that
\begin{equation}\label{Eq:Y}
\{- Z^H(a_k,\widehat{\t}_k), \t \in \R^d\} \stackrel{d}{=} \{a_k^{H_k}  Y^H(\s), \s \in \R^{d-1}\},
\end{equation}
where $Y^H = \{Y^H(\s), \s \in \R^{d-1}\}$ is a $(d-1, 1, \alpha)$-LFSS, whose kernel function
is similar to the function $g_H$ in (\ref{Eq:functh}), but without the factor in $t_k$ and a different
constant $\kappa$. Hence, similarly to (\ref{Eq:Maximom}), we have
 \begin{equation}\label{Eq:Maximom2}
\E \bigg(\sup_{\s \in [\widehat{\bf a}_k, \, \widehat{\bf b}_k]} \big|Y \big([\widehat{\bf a}_k,\s] \big)\big|^\gamma \bigg)
\le C\, \prod_{j\ne k} (b_j - a_j)^{H_j\gamma}.
\end{equation}
By combining \eqref{increment-1},  \eqref{Eq:Y}, (\ref{Eq:Maximom2}) and \eqref{Eq:Maximom}, we obtain
\begin{equation}\label{Eq:Moricz3}
\begin{split}
&\E\bigg(\sup_{t \in T} \big|Z^H(\t) - Z^H({\bf a})\big|^\gamma \bigg) \le  C\bigg(\prod_{j=1}^d (b_j - a_j)^{H_j\gamma}
+ \sum_{k=1}^d a_k^{H_k \gamma} \prod_{j \ne k} (b_j - a_j)^{H_j \gamma}\bigg).
\end{split}
\end{equation}
This proves \eqref{Eq:Moricz2}.
\end{proof}

\begin{theorem}\label{Thm:LFSS}
Assume that $1 < \alpha < 2$ and  $ 1/\alpha < H_j < 1$ for every $j
= 1, \ldots, d$. Then for any $\eps > 0$,
\begin{equation}\label{Eq:LFSS}
\lim_{\|\t\|\to \infty} \frac{Z^H(\t)} {\prod_{j=1}^d (1+ |t_j|^{H_j})
\big(\log (1 + |t_j|\big)^{\frac 1 \a + \eps}} = 0
\qquad {\rm a.s.}
\end{equation}
\end{theorem}

\begin{proof} For simplicity we assume $\t \in (1, \infty)^d$ in (\ref{Eq:LFSS}).
Otherwise, the proof can be slightly modified to cover the cases where $\|\t\|\to \infty$
but $t_j \in (0, 1]$ for some $j \in \{1, \ldots, d\}$.


For any $\n \in \mathbb{N}_0^d$, let $\xi (\n) = Z^H([\n,\, \n + \langle 1\rangle ])$, which is the increment of
$Z^H$ over the unit cube with lower-left vertex $\n$. Then $\{\xi (\n), \n \in \N_0^d\}$ is a stationary
$\alpha$-stable random field.

Now let $\eps >0$ be a constant.  We take  $p \in (0, \alpha)$ such that
$p (\alpha^{-1} + \eps) > 1$. To apply Theorem \ref{main-thm2}, we take $\varphi_j (t_j) = (1+ |t_j|^{H_j})
\big(\log (1+ |t_j|)\big)^{\frac 1 \a + \eps}$ ($j = 1, \ldots, d$) and let $a \ge 2$ be an integer that satisfies
the condition in Theorem \ref{main-thm2}.

For any $\m \in \N_0^d$ and $\n \in \N^d$,  notice that the partial sum $S(\m, \n)$ of
$\{\xi (\n), \n \in \N_0^d\}$ over the rectangle
$(\m, \m + \n]$ equals $Z^H\big([\m+\langle 1\rangle , \m + \n]\big)$.
By making use of the representation (\ref{Eq:Rep-LFSS}), we derive (see e.g., Lemma 18 of \cite{ARX08})
that for any $\m \in \N_0^d$,
\begin{equation}\label{Eq:LFS1}
\begin{split}
\E \Big(\big|S(\m; \, \langle a^{n_j} \rangle)\big|^p \Big)
&= \E\Big( \big| Z^H([\m + \langle 1 \rangle, \, \m
+  \langle a^{n_j} \rangle])\big|^p \Big) \\
&= C_{\alpha, p} \, a^{p \sum_{j=1}^{d} n_j H_j},
\end{split}
\end{equation}
where $C_{\alpha, p} = \E\big(\big|Z^H(\langle 1 \rangle)\big|^p\big)$.
It follows from (\ref{Eq:LFS1}) that 
\begin{eqnarray*}
\sup_{\m \in \N_0^d} \E \Bigg(\frac{\big|S(\m; \, \langle a^{n_j} \rangle)\big|^p }
{ \prod_{j=1}^d\varphi_j (a^{n_j})^p}\Bigg)\le C_{\alpha, p}\prod_{j=1}^d  \frac{1}{\big(\log(1+a^{n_j})\big)^{p(\alpha^{-1}+ \eps)}}.
\end{eqnarray*}
This and the fact that $p (\alpha^{-1} + \eps) > 1$ imply
$$
 \sum_{ \n \in \N_0^d} \sup_{\m \in \N_0^d}\E \Bigg(\frac{\big|S(\m; \, \langle a^{n_j} \rangle)\big|^p }
{ \prod_{j=1}^d\varphi_j (a^{n_j})^p}\Bigg) < \infty.
$$
Hence by Theorem \ref{main-thm2}, we get,
\begin{eqnarray} \label{an-limit0}
\lim_{\|\n\|\to \infty}\frac{Z^H(\n) }{\prod_{j=1}^d \varphi_{j} (n_j)}
= \lim_{\|\n\|\to \infty}\frac{S(\0; \n)}{
\prod_{j=1}^d \varphi_{j} (n_j)}  =  0, \quad \hbox{a.s.}
\end{eqnarray}


\noindent Next we show that the last result holds for continuous time
$\t \in (1, \infty)^d$ as $\|\t\|\to \infty$. To this end, let
$T_n =\prod_{j=1}^d [a^{n_j}, a^{n_j+1}] = [a^\n,\, a^{\n + \langle 1 \rangle}]$ and let $\eta>0$ be any fixed constant.
We take $u = \eta \, \prod_{j=1}^d \varphi_j (a^{n_j})$.
Then (\ref{Eq:maxtail}) implies
\begin{equation}\label{Eq:MaxIn4}
\begin{split}
\P\bigg(\sup_{\t \in T_n} \big|Z^H (\t)- Z^H  (a^{\n}) \big| \geq
\eta \prod_{j=1}^d \varphi_j (a^{n_j})\bigg) &\leq
C \prod_{j=1}^d \bigg(\frac{ a^{n_j H_j} } {\varphi_j (a^{n_j})}\bigg)^{\gamma}   \\
&\leq  C \prod_{j=1}^d  \frac{1} {\big(n_j \log a \big)^{\gamma(\alpha^{-1}+\eps)}}.
\end{split}
\end{equation}
By taking $\gamma < \alpha$ such that $\gamma(\alpha^{-1}+\eps) >1$, we see that the
probabilities in \eqref{Eq:MaxIn4} are summable. Hence the Borel-Cantelli lemma yields
 \begin{equation}\label{t-limit1}
\limsup_{\|\n\| \rightarrow \infty} \frac{\sup_{\t \in T_n}
\big|Z^H(\t)-Z^H(a^{\n})\big|}{\prod_{j=1}^d \varphi_j (a^{n_j}) }
=0, \quad \hbox{a. s.}
\end{equation}
Finally,  for any $t \in (1, \infty)^d$ with $\|\t\|$ large, there is an
$\n \in \N_0^d$ such that $\t \in [a^{\n}, a^{\n + \langle 1 \rangle}]$.
Since
\[
\frac{\big|Z^H(\t)\big|}{\prod_{j=1}^d \varphi_j (t_j)}
\le \frac{\big|Z^H(a^{\n})\big|} {\prod_{j=1}^d \varphi_j (a^\n)}
+ \frac{\sup_{\t \in T_n} \big|Z^H(\t)-Z^H(a^{\n})\big|}{\prod_{j=1}^d \varphi_j (a^{n_j})},
\]
it is clear that (\ref{Eq:LFSS}) follows from (\ref{an-limit0}) and (\ref{t-limit1}). This proves Theorem \ref{Thm:LFSS}.
\end{proof}


Similar results holds for harmonizable fractional stable sheets 
and other fractional stable fields in \cite{Xiao10,Xiao08}. We leave 
the details to an interested reader. 

\end{document}